\documentclass[12pt]{amsart}
\usepackage{url,  mathrsfs}
\usepackage{hyperref}
\usepackage{amsmath, amssymb}
\usepackage{graphicx, 
epstopdf}
\usepackage{subfig}
\usepackage{color}
\usepackage{bbm}
\usepackage{subfloat}
\usepackage[draft]{fixme}
\usepackage{bm} 
\usepackage{bbm} 
\usepackage{epigraph}
\usepackage{hyperref}
\usepackage{endnotes}
\usepackage{ifpdf}
\usepackage{array}
\usepackage[abs]{overpic}
\usepackage{pict2e}

\usepackage{multicol}
\usepackage{multirow}
\usepackage{graphicx}
\usepackage{tikz}
\usepackage{mathdots}

\newcommand{\N}{\mathbb{N}}
\newcommand{\Z}{\mathbb{Z}}
\newcommand{\R}{\mathbb{R}}

\renewcommand{\S}{\mathcal{S}}

\newcommand{\x}{\mathbf{x}}
\newcommand{\p}{\mathbf{p}}

\theoremstyle{plain}
\newtheorem{thm}{Theorem}[section]
\newtheorem{mainthm}{Main Theorem}
\newtheorem{prop}[thm]{Proposition}
\newtheorem{cor}[thm]{Corollary}
\newtheorem{lem}[thm]{Lemma}
\newtheorem{question}[thm]{Question}

\newtheorem*{thm*}{Theorem} 

\theoremstyle{definition}
\newtheorem{Def}[thm]{Definition}
\newtheorem{example}[thm]{Example}
\newtheorem{remark}[thm]{Remark}
\newtheorem{fact}[thm]{Fact}

\DeclareMathOperator{\pop}{p}

 \usepackage[inner=1in,outer=1in,top=1in,bottom=1in]{geometry}
 \usepackage{breqn}

\begin{document}
\title[Sparse moments of step functions and allele frequency spectra]{Sparse moments of univariate step functions\\ and allele frequency spectra}

 \author{Zvi Rosen}
  \address{Florida Atlantic University,  Boca Raton, FL, USA}
\email{rosenz@fau.edu}
 \author{Georgy Scholten}
 \address{North Carolina State University, Raleigh, NC, USA }
\email{ghscholt@ncsu.edu}
 \author{Cynthia Vinzant}
 \address{North Carolina State University, Raleigh, NC, USA}
\email{clvinzan@ncsu.edu}

\begin{abstract}
We study the univariate moment problem of 
piecewise-constant density functions on the interval $[0,1]$ 
and its consequences for an inference problem in population genetics. 
We show that, up to closure, any collection of $n$ moments is achieved by a step function with at most $n-1$ breakpoints and that this bound is tight. 
We use this to show that any point in the $n$th coalescence manifold in population genetics can be attained by a piecewise constant population history with at most $n-2$ changes.
Both the moment cones and the coalescence manifold are projected spectrahedra and we describe the problem of finding a nearest point on them as a semidefinite program. 
  \end{abstract}

\maketitle

Given a finite collection $A\subset \N$, we consider the convex cone $M(A)$ of all moments $(m_{a})_{a\in A}$ 
of the form $m_a = \int x^a d \mu$ where $\mu$ is a nonnegative Borel measure on the unit interval $[0,1]$. 
For consecutive moments $A = \{0,1,2\hdots, d\}$, this is a classical object in analysis and real algebraic geometry. The problem of determining membership in the cone $M(A)$ is known as the truncated Haussdorff moment problem. 
See, for example,  \cite{ConvexAlgBook, LasserreBook, PosPoly, Sch}.

In this paper we study moments coming from piecewise-constant density functions with the idea of minimizing the number of pieces needed. 
Formally, we consider the set $M_k(A)$ as the closure of the set of moments $(m_{a})_{a\in A}$ where $m_a = \int_0^1x^af(x)dx$ and $f$ is a 
nonnegative step function with at most $k$ discontinuities. Our main theorem is the following:

\begin{mainthm}\label{mainthm:StepBound}
$M_k(A) = M(A)$ if and only if $k\geq |A|-1$. 
\end{mainthm}
This is the content of Theorem~\ref{thm:MomentBound1} and  Corollary~\ref{cor:MomentBound2}. 
The proof involves studying the convex algebraic boundary of these cones and in particular showing that they are simplicial (Corollary~\ref{cor:simplicial}). 
When restricting to the moments of monotone density functions, 
only half as many break points are needed (see Propositions~\ref{prop:MonotoneCone} and \ref{prop:monotoneLowerBound}). 

 \begin{mainthm}\label{mainthm:Monotone}
Every $A$-moment vector of a monotone density function is the 
 limit of $A$-moments of monotone step functions with $\leq k$ breakpoints 
 if and only if  $ k \geq \lfloor |A|/2 \rfloor$.  
\end{mainthm}

One of our motivations for studying this problem came from its relation to the coalescence manifold studied by \cite{BRS}. 
The coalescence manifold $\mathcal{C}_{n,k}$, formally defined in Section~\ref{sec:coalManifold}, is a set of summary statistics in population genetics,
derived from observing $n$ genomes with a population history consisting of $k+1$ different population sizes.  
Our last main theorem, appearing as Theorem~\ref{thm:coal}, is that the 
coalescence manifold $\mathcal{C}_{n,k}$ coincides with an affine section of the moments $M_{k}(A)$ 
for $A = \{0, 2, \hdots, \binom{n}{2} -1\}$. 

\begin{mainthm}
The coalescence manifold $\mathcal{C}_{n,k}$ is the intersection of 
$M_k(A)$ with the affine hyperplane of points with coordinate sum 
equal to one for $A = \{0, 2, \hdots, \binom{n}{2}-1\}$. That is,  
$\mathcal{C}_{n,k}  = \left\{(m_a)_{a\in A}\in M_{k}(A)  :   \sum_{a\in A} m_a = 1\right\}$.
\end{mainthm}

The authors in \cite{BRS} show that the manifold $\mathcal{C}_{n,k}$ stabilizes at 
$k=2n-2$, i.e. $\mathcal{C}_{n,2n-2} = \mathcal{C}_{n,k}$ for all $k\geq 2n-2$. 
Together, the main theorems above improve this bound by a factor of two, showing that the coalescence manifolds stabilize at $k = n-2$ and this bound is tight. 

The connection with the moment problem also provides a description of $ \mathcal{C}_{n,n-2}$ as the projection of a spectrahedron. 
The problem of finding the nearest point in $ \mathcal{C}_{n,n-2}$ to a given point in $\R^{n-1}$ can then be formulated as a semidefinite program.

This paper is organized as follows. In Section~\ref{sec:momentIntro}, we introduce formal definitions 
of the moment sets $M_k(A)$, study their convex and algebraic structure, and prove Theorem~\ref{mainthm:StepBound}. 
In Section~\ref{sec:Monotone}, we analyze analogous questions for moment problems coming from monotone step functions.  
The definitions and connections with the coalescence manifold $\mathcal{C}_{n,k}$ are given in Section~\ref{sec:coalManifold}.
Semidefinite descriptions of these sets are discussed in Section~\ref{sec:sdp}. 
Finally we end with a discussion of open problems surrounding these interesting sets in Section~\ref{sec:questions}.

\subsection*{Acknowledgements}
Author Rosen thanks Dr.~Yun S.~Song for introducing him to this topic. 
Authors Scholten and Vinzant were partially supported by NSF-DMS grants \#1620014 and \#1943363. 
This material is based upon work directly supported by the National Science Foundation Grant No. DMS-1926686, and indirectly supported by the National Science Foundation Grant No. CCF-1900460.

\section{Moments of step functions}\label{sec:momentIntro}
For $k\in \N$, let $S_k$ denote the set of nonnegative step functions on $[0,1]$ of the form 
\begin{equation}\label{eq:stepFunction}
f = y_1{\bf 1}_{[0,s_1]} +\sum_{i=2}^{k+1} y_i {\bf 1}_{(s_{i-1},s_{i}]},
\end{equation}
where $0=s_0< s_1< \hdots <s_k<s_{k+1}=1$ and $y_1, \hdots, y_{k+1}\in \R_{\geq 0}$.
Note that:
\begin{enumerate} \item $S_k$ is invariant under nonnegative scaling,
  \item $S_k\subseteq S_{\ell}$ when $k \leq\ell$,
    and \item $S_{k} + S_{\ell}$, defined as $\{f + g \: \mid \: f \in S_k, g\in S_\ell\}$, is a subset of $ S_{k+\ell}$.
    \end{enumerate}

Elements of $S_k$ define nonnegative measures on $[0,1]$. We will be interested in the possible moments of these measures. 
Given a finite collection $A\subset \N$, we define to be 
the Euclidean closure of the set moments given by density functions in $S_k$: 
\[
M_k(A) \ =\ \overline{\left\{ \left(\int_0^1x^af(x)dx\right)_{a\in A} : f\in S_k\right\}}.
\]

One important case is that of consecutive moments $A = \{0,1,\hdots, d\}$. 
For any finite collection $A\subset
\N$, the moment cone $M_k(A)$ can be expressed, up to closure, as the
image of $M_k(\{0,1,\hdots, \max(A)\})$ under the coordinate projection
$\pi_A:\R^{\max(A)+1}\to \R^A$ given by $\pi_A(m_0, \hdots,
m_{\max(A)}) = (m_a)_{a\in A}$. 

\begin{remark}\label{rem:sum}
By linearity of the integral, we see that $M_k(A)$ inherits many properties of $S_k$. 
That is, $M_k(A)$ is invariant under nonnegative scaling,
$M_k(A)\subseteq M_{\ell}(A)$ when $k \leq\ell$ 
and $M_{k}(A) + M_{\ell}(A) \subseteq M_{k+\ell}(A)$ (here, in the sense of
the Minkowski sum), as desired. 
\end{remark}

We will be interested in comparing this to the full moment cone: 
\[
M(A)  \ =\ \overline{\left\{ \left(\int_0^1x^a d\mu \right)_{a\in A} :\mu \text{ is a nonnegative Borel measure on }[0,1]\right\}}.
\]
The cone $M(A)$ is dual to the convex cone of univariate polynomials 
supported on $A$ that are nonnegative on $[0,1]$, as will be discussed below in
Proposition~\ref{prop:boundaryMoments}.

When $0\in A$, the closure in the definition of $M(A)$ is not necessary, and 
the extreme rays of $M(A)$ are come from point evaluations. 
That is, we can write $M(A)$ as the conical hull of the image of $[0,1]$ under the corresponding moment map: 
\[
M(A) \ = \  {\rm conicalHull}\left\{v_A(t) : t\in [0,1]\right\}
\ \ \text{ where } \ \ 
v_A(t) = (t^a)_{a\in A}.
\]
See, for example, \cite[Prop. 10.5]{Sch}.

When $0\notin A$, this equality only holds up to closure, as the curve 
parametrized by $v_A(t)$ includes the origin. In this case, 
$M(A)= \overline{{\rm conicalHull}\left\{v_A(t) : t\in [0,1]\right\}}$. 
As we will see below, then we can still write $M(A)$ as the conical 
hull of a curve segment. Specifically, 
 $M(A)= {\rm conicalHull}\left\{v_B(t) : t\in [0,1]\right\}$ where $B = \{ a-\min(A) : a\in A\}$.

\begin{lem}\label{lem:moment_shift}
If $A\subset \N$ is finite and $B = \{ a-\min(A) : a\in A\}$, then
$M(A) = M(B)$. 
\end{lem}
\begin{proof}
For $t\in (0,1]$, the point $v_A(t)$ can be rewritten as
$t^{\min(A)}v_B(t)$, a scalar multiple of $v_B(t)$.
It follows that the conical hulls of $\left\{v_A(t) : t\in (0,1]\right\}$
and $\left\{v_B(t) : t\in (0,1]\right\}$ are equal.
We observe that the extreme ray $v_B(0)$ of $M(B)$ can be attained in
the closure of $M(A)$ as
the limit of the moment of the step function
$f=\epsilon^{-(\min(A)+1)}{\bf 1}_{[0,\epsilon]}$ as $\epsilon$ goes
to zero:
\[\lim_{\epsilon \rightarrow 0}\int_{0}^1 x^a f(x)dx 
\ = \ \lim_{\epsilon \rightarrow 0} \epsilon^{-(\min(A)+1)}\int_{0}^{\epsilon}x^a dx
\ = \ \lim_{\epsilon \rightarrow 0} \frac{\epsilon^{a-\min(A)}}{a+1} \\
\ =  \                     \begin{cases}
                          \frac{1}{a+1} & \text{if } a= \min(A)\\
                          \ 0 & \text{otherwise.} 
                        \end{cases}
\]
It follows that $v_B(0)$ belongs to $M(A)$. Since $M(B)$ can be
written as the union of the cone over $v_A(t)$ for $t \in (0,1]$ and
the ray over $v_B(0)$, the equality between the two cones ensues: 
\[
M(A)=  \overline{{\rm conicalHull}\left\{v_A(t) : t\in [0,1]\right\}}
= M(B). \qedhere
\]
\end{proof}

\begin{example}
Consider $A = \{1,2\}$ and $B = \{0,1\}$. Then 
\[
M(B) = {\rm conicalHull}\{(1,t) : t\in [0,1]\}  = \overline{{\rm conicalHull}\{(t,t^2) : t\in [0,1]\}} = M(A).
\]
Here we see the need for taking closures when $0\not\in A$. The point $(1,0)= v_B(0)$ is not contained in the 
the conical hull of the curve segment $\{(t,t^2) : t\in [0,1]\}$ but is contained in its closure. See Figure~\ref{fig:A={1,2}}. In this case, the boundary of $M(A)= M(B)$ consists of scalar multiples of $v_B(0)$ and 
$v_B(1)$, both of which belong to $M_1(A)$, by Proposition~\ref{prop:shiftDirac=2} below.
Arguments below will then show that $M(A) = M_1(A)$. 
\end{example}

\begin{figure}
\includegraphics[height=2in]{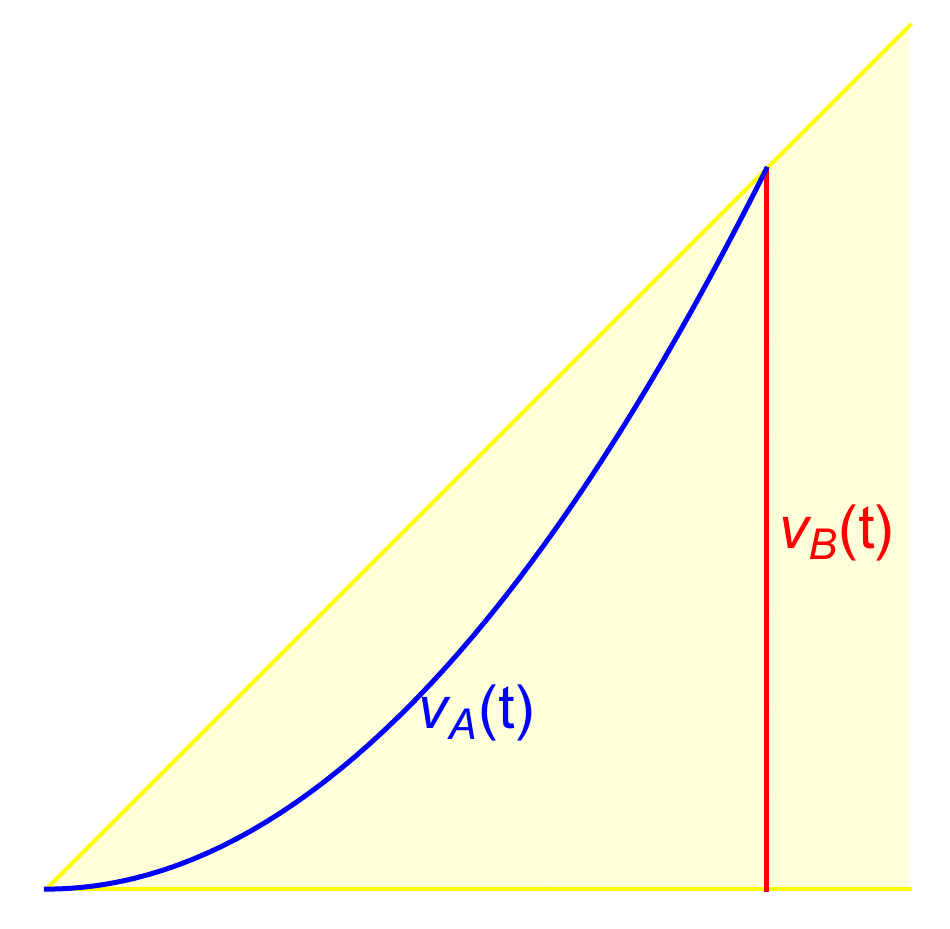}
\caption{The cone $M(A)$ along with the curve segments $v_{A}([0,1])$
  and $v_{B}([0,1])$ for $A=\{1,2\}$ and $B = \{0,1\}$.}
\label{fig:A={1,2}}
\end{figure}

\begin{prop}\label{prop:shiftDirac=2}
Let $A\subset \N$ be finite and let $B = \{ a-\min(A) : a\in A\}$. 
The points $v_B(0)$and $v_B(1)$ belong to
$M_1(A)$ and for every $t\in (0,1)$, $v_B(t)$ belongs to $M_2(A)$.
\end{prop}
\begin{proof}
For $0< t<1$, and $0<\epsilon < 1-t$ consider the step function 
$f = \epsilon^{-1} t^{-\min(A)} {\bf 1}_{(t,t + \epsilon]}$ in $S_2$.  
By continuity, the integral $\int_0^1x^af(x)dx$ limits to $t^{a-\min(A)}$ as
$\epsilon \to 0$, thus $M_2(A)$ contains the limit point $v_B(t) =
(t^b)_{b\in B}$. Similarly,  the limit as $\epsilon\to 0$ of the $A$-moment vectors  
of step functions ${f = (\min(A)+1)\epsilon^{-1-\min(A)}   {\bf 1}_{[0, \epsilon]}}$
and ${f = \epsilon^{-1}  {\bf 1}_{(1-\epsilon, 1]}}$ in $S_1$ 
are $v_B(0)$ and $v_B(1)$, respectively. Therefore these vectors belong to $M_1(A)$. 
\end{proof}

A corollary of this statement is that $ M_k(A)=M(A)$ for $k = 2|A|$. 
By Carath\'eodory's Theorem, any point in $M(A)$ is in the
conical hull of at most $|A|$ points of the 
form $(t^a)_{a\in A}$ where $t\in [0,1]$, each of which 
belongs to $M_2(A)$ by Proposition~\ref{prop:shiftDirac=2}. 
By Remark~\ref{rem:sum}, the sum of 
$|A|$ elements from $M_{2}(A)$ belongs to $M_{2|A|}(A)$, giving 
$M(A)\subseteq M_{2|A|}(A)$. 
In fact, $M_k(A)$ fills out the whole moment cone much sooner: 

\begin{thm}\label{thm:MomentBound1}
If $k\geq |A|-1$, $ M_k(A)=M(A)$.
\end{thm}

The proof of this theorem relies on understanding the points on the boundary of $M(A)$. 

\begin{prop}\label{prop:boundaryMoments}
Let $A\subset \N$ be finite with $0\in A$. 
If ${\bf m} = (m_a)_{a\in A}$ belongs to the Euclidean boundary of $M(A)$, 
then any representing measure $\mu$ on $[0,1]$ with $m_a = \int x^a d\mu$
has finite support. Specifically, the support of $\mu$ is a subset of
the roots contained in $[0,1]$ of a polynomial nonnegative on $[0,1]$ and of the form $p(x) = \sum_{a\in A} p_a x^a$.
The vector ${\bf m}$ is a conic combination of 
the vectors $v_A(r)$ where $r$ ranges over the roots of $p$. 
\end{prop}
\begin{proof}
Let $\ell:\R^{A}\to \R$ be a linear function $\ell({\bf v}) = \sum_{a\in A} p_a v_a$ 
defining a supporting hyperplane of $M(A)$ 
at ${\bf m}$. That is, $\ell({\bf v})\geq 0$ for all ${\bf v}\in M(A)$ and 
$\ell({\bf m})=0$.  Consider the polynomial  $p(x) = \ell(v_A(x)) = \sum_{a\in A} p_a x^a$. 
Since $v_A(t)\in M(A)$ for all $t\in [0,1]$, $p$ is nonnegative on $[0,1]$. 
Furthermore, for any measure $\mu$ with moments ${\bf m}$, 
\[
\int p(x) d \mu  = \sum_{a\in A} p_a m_a = \ell({\bf m}) = 0. 
\]
The measure $\mu$ is nonnegative and the polynomial $p$ is nonnegative on $[0,1]$. 
From this we see that the support of the measure $\mu$ must be contained in 
the (finite) set of roots $R$ of $p(x)$.  Specifically, $\mu = \sum_{r\in R} w_r \delta_{r}$
for some $w_r\in \R_{\geq 0}$; therefore, ${\bf m} = \sum_{r\in R} w_r v_A(r)$.
\end{proof}

\begin{proof}[Proof of Theorem~\ref{thm:MomentBound1}]
First, consider a point ${\bf m}$ in the boundary of $ M(A)$. 
By Lemma~\ref{lem:moment_shift}, $M(A) = M(B)$ where $B =  \{ a-\min(A) : a\in A\}$, 
and so ${\bf m}$ also belongs to the boundary of $ M(B)$.
By Proposition~\ref{prop:boundaryMoments}, 
${\bf m}$ is the vector of $B$-moments of a measure 
 $\mu$ supported on the roots of 
 a nonnegative polynomial on $[0,1]$ of the form $p(x) = \sum_{a\in B} p_a x^a$.
Let $b$ be the number of distinct roots of $p$ in the set $\{0,1\}$ and 
$i$ be the number of distinct roots in of $p$ in the open interval $(0,1)$. 
Then ${\bf m}$ is in the conical hull of the $b+i$ points 
given by $v_B(r)$ where $r$ ranges over these roots. 
By Proposition~\ref{prop:shiftDirac=2}, 
${\bf m}$ belongs to $M_{k}(A)$ for $k=b+2i$. 

By Descartes' rule of signs, the number of positive roots of $p$, 
counting multiplicity, is at most the number of sign changes in the
list of coefficients 
$\{p_a\}_{a\in B}$. If $p_0 \neq 0$, then $p$ has at most $|B|-1$ roots in $\R_{>0}$. 
If $p_0 = 0$,  then $p$ is the sum of at most $|B|-1$ nonzero terms and its
number of roots in $\R_{>0}$ must be smaller or equal to $|B|-2$.  
Note that every root of $p$ in $(0,1)$ must have even multiplicity greater or
equal to $2$.  All together this gives $b+2i\leq |B|-1= |A|-1$. 

Now consider ${\bf m}$ in the interior of $M(A)$. 
Let ${\bf c} = (1/(a+1))_{a\in A} \in M_0(A)$ denote the vector obtained by 
integrating against the constant step function of height one.  
Let $\lambda^*$ be the maximum value of $\lambda\in \R$ 
for which ${\bf m} - \lambda {\bf c}$ belongs to $M(A)$. 
From ${\bf m}\in M(A)$, we see that $\lambda^*\geq 0$. 
Moreover, since $M(A)$ is pointed, $-{\bf c}$ does not belong to $M(A)$,
meaning that for sufficiently large $\lambda$, ${\bf m} - \lambda {\bf
  c}$ does not belong to $M(A)$.
Since $M(A)$ is closed, it follows that such a maximum $\lambda^*$ must exist.  

The point ${\bf m} - \lambda^* {\bf c}$ belongs to the boundary of $M(A)$.
By the arguments above, ${\bf m} - \lambda^* {\bf c}$ belongs to 
$M_{k}(A)$ for $k\geq  |A|-1$.  Since ${\bf c}\in M_0(A)$ and 
\[
{\bf m}  = ({\bf m} - \lambda^* {\bf c}) + \lambda^*{\bf c}, 
\]
the point ${\bf m}$ also belongs to $M_{k}(A)$ for $k\geq  |A|-1$. 
\end{proof}

\begin{remark}
 It follows from the proof of Theorem~\ref{thm:MomentBound1} that for
  all $k\geq 0$, $M_k(A)$ is star convex with respect to the point 
  ${\bf c} =(1/(a+1))_{a\in A} $, the $A$-moment of the constant function.
  Indeed, since ${\bf c}$ belongs to $M_{0}(A)$, 
  $\lambda {\bf c} + M_k(A) \subseteq M_k(A)$ for all $\lambda \geq 0$.  
\end{remark}

We can go further in characterizing the facial structure of the boundary
of $M(A)$. Through a connection to Schur polynomials, we can deduce
linear independence among sets of points from the curve of the correct size.

\begin{prop} \label{prop:SchurDetPos}
For a collection $A$ of  integers $0= a_1<a_2 < \hdots< a_n$
and any real values $0\leq r_1 < r_2 < \hdots < r_{n} \leq 1$, the
determinant of the matrix $\S_A$ is strictly positive, where 
\[
\S_A({\bf r})= \begin{pmatrix}
          1& 1 & \ldots  & 1 \\
          r_1^{a_2}& r_2^{a_2}  & \ldots  & r_n^{a_2} \\
        \vdots & \ \vdots & & \vdots \\
	 r_1^{a_n} & r_2^{a_n} & \ldots & r_n^{a_n}
      \end{pmatrix}.
\]
\end{prop}
\begin{proof}
By the bialternant formula for Schur polynomials, 
the determinant of the matrix $\S_A({\bf r})$ can be
expressed as
\begin{equation}\label{eq:schur}
  \det (\S_A)= \left(\prod_{1\leq i<j\leq
    n}(r_j-r_i)\right)\mathfrak{s}_{\lambda}(r_1,\ldots,r_n) \text{
  for } \lambda = (a_n-(n-1),a_{n-1}-(n-2),\ldots, a_1),
\end{equation}
where $\mathfrak{s}_{\lambda}(x_1, \hdots, x_n)$ denotes the Schur
polynomial associated to the partition $\lambda$. By definition, the
Schur polynomial
$\mathfrak{s}_\lambda(x_1, \hdots, x_n)$ is the sum of monomials ${\bf
  x}^T$ over all semistandard Young tableaux $T$ of
shape $\lambda$. One can observe, either from expanding the
determinant of $\S_A(r)$ along the first column, or by filling out a
semistandard Young Tableau of shape $\lambda$ without using the number
$1$, that $x_1$ does not appear in all the monomials of the
determinant of
$\S_A$. It follows that $\det (\S_A)$ is strictly positive for any
$0\leq r_1 < r_2 < \hdots < r_{n} \leq 1$.
\end{proof}

\begin{cor}\label{cor:simplicial}
All proper faces of $M(A)$ are simplicial. 
\end{cor}
\begin{proof}
By Lemma~\ref{lem:moment_shift}, we can assume that $0\in A$. 
Recall that $M(A)$ is the conical hull over the curve segment $\{v_{A}(t):t\in
[0,1]\}$ and any proper face $F$ of this cone can be expressed as the
conical hull of some points $v_{A}(r_1), \hdots, v_{A}(r_k)$ where 
$0\leq r_1 < r_2 < \hdots < r_{k} \leq 1$.  
If $k> \dim(F)$, then there is a subset of these points of size $\dim(F)+1\leq n$, which 
necessarily lie in $F$ and are therefore linearly dependent,
contradicting Proposition~\ref{prop:SchurDetPos}.  Therefore $k =
\dim(F)$ and $F$ is simplicial.
\end{proof}

This lemma lets us assign an \emph{index} to points on the boundary of $M(A)$, following \cite[Ch. 10.2]{Sch}. 
Let ${\bf m}$ be a point on the boundary of $M(A)$. By Corollary~\ref{cor:simplicial}, 
there is a \emph{unique} representation of ${\bf m}$ as $\sum_{j=1}^k w_j v_{A}(r_j)$ 
where $0\leq r_1 < \hdots < r_k\leq 1$ and $w_1, \hdots, w_k\in \R_{>0}$. 
We define the \textbf{index} of ${\bf m}$, denoted ${\rm ind}({\bf m})$, to be 
$b+2i$ where $b = \#\{ j : r_j\in \{0,1\}\}$ and $i = \#\{ j : r_j\in (0,1)\}$.
By Proposition~\ref{prop:shiftDirac=2}, any point ${\bf m}$ on the boundary of $M(A)$ 
belongs to $M_{{\rm ind}({\bf m})}(A)$.

To prove the converse, we must rule out the
possibility that ${\bf m} \in M_k(A)$ for $k < {\rm ind}({\bf m})$. In other words, 
it is impossible to approach a point ${\bf m}$ on the boundary of $M(A)$ 
with moment vectors of step functions with fewer breakpoints than expected. 
\begin{lem}\label{lem:useful}
Let ${\bf m}$ be a point on the boundary of $M(A)$. For $k<{\rm ind}({\bf m})$, ${\bf m}\not\in M_k(A)$. 
That is, if ${\bf m}\in M_k(A)$, then ${\rm ind}({\bf m})\leq k$.
\end{lem}
\begin{proof}
Note that for any non-zero point ${\bf m}$ in $M(A)$, $m_0>0$ and so we can rescale 
${\bf m}$ to have $m_0 = 1$.  We will 
write $M_k(A) \cap \{m_0=1\}$ as the image of a compact polytope under a polynomial map
and check that any point ${\bf m}$ in the image of this map and the boundary of $M(A)$ 
has index $\leq k$. 

Any function $f \in S_k$ can be written as 
$f = y_1 {\bf 1}_{[0,s_1]} +\sum_{i=2}^{k+1} y_i {\bf 1}_{(s_{i-1},s_{i}]}$ 
for some values $0=s_0 < s_1 < \ldots < s_k < s_{k+1}=1$ 
and $y_i\geq 0$ for all $i$. 
We now introduce transformed $w$-coordinates by letting $w_i = y_i(s_i - s_{i-1})$
denote the area $\int_{s_{i-1}}^{s_i}f(x)dx$. The corresponding moment in $M_k(A)$ 
is given by the image of the point $({\bf s}, {\bf w})  = (s_1, \hdots, s_k, w_1, \hdots, w_{k+1})$ under 
the polynomial map
\begin{equation}\label{eq:muA}
\mu_A({\bf s}, {\bf w}) = \left(\sum_{i = 1}^{k+1} y_i
\dfrac{s_{i}^{a+1}-s_{i-1}^{a+1}}{a+1}\right)_{a \in A} =
\left(\sum_{i = 1}^{k+1} w_i \dfrac{(s_{i}^{a} + s_i^{a-1}s_{i-1} + \cdots + s_{i-1}^{a})}{a+1}\right)_{a \in A}. 
\end{equation}
Note that the constraint that $m_0=1$ translates into $\sum_i w_i=1$. 
Consider the polytope 
\begin{equation}\label{eq:parameterPolytope}
P = \left\{({\bf s}, {\bf w})\in \R^{k}\times \R^{k+1} \ \text{ such that } \ 
0 \leq s_1 \leq \ldots \leq s_k \leq 1, w_i \geq 0, \  \sum_{i=1}^{k+1} w_i = 1\right\},
 \end{equation}
which is a product of two simplices of dimension $k$. 
The moments of step functions $f\in S_k$ with $\int f(x) dx = 1$
is the image under $\mu_A$ of the set of points $({\bf s}, {\bf w})\in P$ 
with distinct $0<s_1<\hdots < s_k<1$. 
Its closure  is $M_k(A)\cap \{m_0=1\}$, which necessarily coincides with 
the image of $P$ under $\mu_A$, as the image of 
a compact set under a  continuous  map is closed.

If $w_i>0$ and $s_{i-1}<s_{i}$ for some $i$, then $\mu_A({\bf s}, {\bf w})$ 
has a representing measure whose support includes the interval $(s_{i-1},s_i]$ 
and is therefore not finite. Then by Proposition~\ref{prop:boundaryMoments}, 
${\bf m}$ belongs to the  interior of $M(A)$.  

Suppose the point ${\bf m}$ belongs to $M_k(A)$. Then ${\bf m} = \mu_A({\bf s}, {\bf w})$
for some $({\bf s}, {\bf w})\in P$. Let $I$ denote the collection of indices $1\leq i\leq k$ for which $w_i>0$. 
If ${\bf m}$ belongs to the boundary of $M(A)$, $s_{i-1} = s_i$ for all $i\in I$. 
Then 
\[
{\bf m} = \mu_A({\bf s}, {\bf w}) = \sum_{i\in I} w_i v_A(s_i). 
\]
We can bound ${\rm ind}({\bf m})$ by bounding the number of distinct values 
of $s_i$ that appear. 
For each $i\in I$ with $s_i\in (0,1)$, $s_i$ equals $s_{i-1}$, hence there are  
at least two indices $j$ in $\{1, \hdots, k\}$ for which $s_j = s_i$. 
Trivially, if $s_i\in \{0,1\}$, there is at least one $j\in \{1,
\hdots, k\}$ such that $s_j = s_i$. 
Together, these show that 
\[
{\rm ind}({\bf m}) \ = \  \#\{s_i \in \{0,1\} : i\in I\} + 2\cdot
\bigl(\#\{s_i \in (0,1) : i\in I\}\bigl) \ \leq \ k. \qedhere
\] 
\end{proof}

\begin{lem}\label{lem:BoundaryDim}
The intersection of 
$M_k(A)$ with the Euclidean boundary of $M(A)$ is a semialgebraic set of 
dimension $\leq k$. 
\end{lem}
\begin{proof} 
By Lemma~\ref{lem:useful}, the intersection of $M_k(A)$ with the Euclidean boundary of $M(A)$
is the set of boundary points of index $\leq k$. We can parametrize this as the union of the semialgebraic sets: 
\[\bigcup_{\sigma\in \{0,1\}^2}\left\{\sum_{j=1}^{\ell} w_j v_{A}(r_j) + w_{\ell+\sigma_1}v_A(0) +w_{\ell+\sigma_1 + \sigma_2}v_A(1) :   {\bf r} \in (0,1)^{\ell}, {\bf w} \in (\R_{>0})^{\ell+\sigma_1 + \sigma_2} \right\},  \]
where in each set, $\ell$ is chosen so that $2\ell  +\sigma_1 + \sigma_2 \leq k$.  
Here we use ${\bf r}$ to denote the vector $(r_j)_j$ and ${\bf w}$ for the vector $(w_j)_j$. 
Note that each set is the image of $(0,1)^n\times  (\R_{>0})^m$ under a polynomial map 
where $n+m\leq k$ and therefore has dimension $\leq k$. 
\end{proof}

\begin{cor}\label{cor:MomentBound2}
If $k< |A|-1$, $ M_k(A)\neq  M(A)$.
\end{cor}
\begin{proof}
The cone $M(A)$ is full-dimensional in $\R^{|A|}$, in consequence, the
cone's boundary is a hypersurface of dimension $|A|-1$. 
By Lemma~\ref{lem:BoundaryDim}, the dimension of the intersection of 
$M_k(A)$ with the boundary of $M(A)$ has dimension $\leq k$, so for $k<|A|-1$, 
this cannot be the entire boundary of $M(A)$. 
\end{proof}


\begin{example} Consider $A = \{0,2,5,9\}$. To visualize the 
moment sets $M_k(A)$, we consider their intersections with the 
affine hyperplane $\{m_0=1\}$.  Affine transformations of these intersections 
 are shown in Figure~\ref{fig:C5ex}. 
Note that the step functions with at most one breakpoint and total mass one 
can be written as 
$\lambda  {\bf 1}_{[0,1]} + (1-\lambda) \frac{1}{s}{\bf 1}_{[0,s]}$ or $\lambda  {\bf 1}_{[0,1]} + (1-\lambda) \frac{1}{1-s}{\bf 1}_{(s,1]}$ where $\lambda\in [0,1]$. 
The result is a two-dimensional surface in the plane $\{m_0=1\}$. 
The set  $M_2(A)$ is full-dimensional, but does not fill up all of  $M(A)$. 
As promised by Lemma~\ref{lem:BoundaryDim}, the intersection $M_2(A)$  with 
the boundary of $M(A)$ has dimension $\leq 2$, so its image in $\{m_0=1\}$ has dimension $\leq 1$.  Indeed, we see this intersection is given by the curve parametrized by $(t^2, t^5, t^9)$ for $t \in [0,1]$
and the line segment between its end points $(0,0,0)$ and $(1,1,1)$. 
Finally, by Theorem~\ref{thm:MomentBound1}, $M_3(A)$ is the full cone $M(A)$.
Points on the boundary of $M(A)$ have index $\leq 3$, and so have one of the two forms 
$w_0 v_A(0) + w_r v_A(r)$ or $w_1 v_A(1) + w_r v_A(r)$ where $r\in [0,1]$, $w_0, w_1, w_r\in \R_{\geq 0}$. 
\end{example}

\begin{figure}[h]
\includegraphics[height=1.5in]{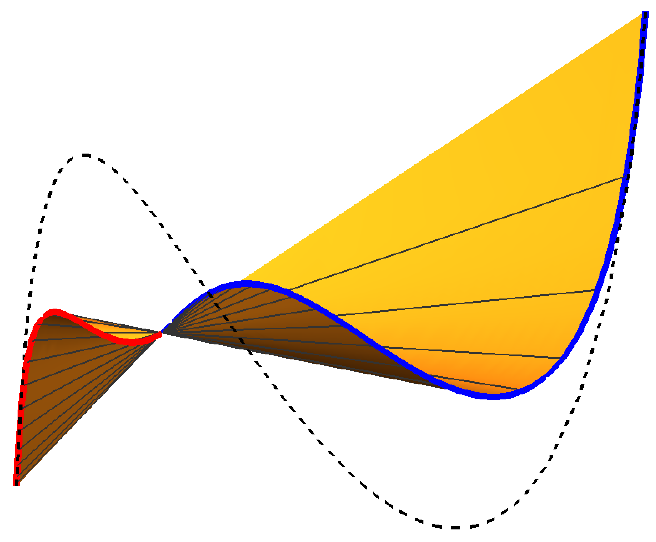} \quad \includegraphics[height=1.5in]{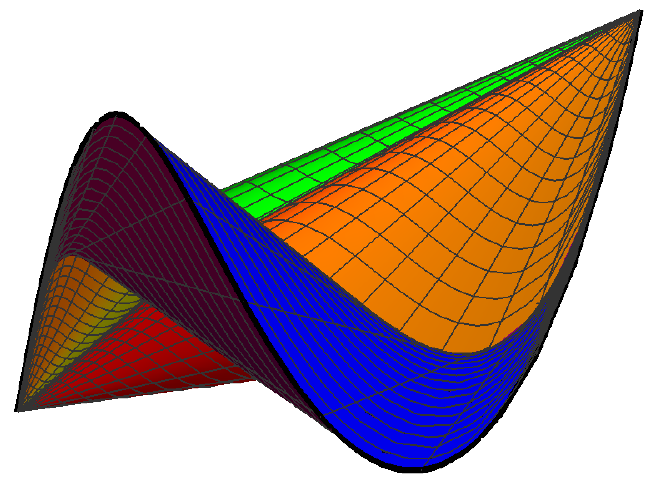} \quad\includegraphics[height=1.5in]{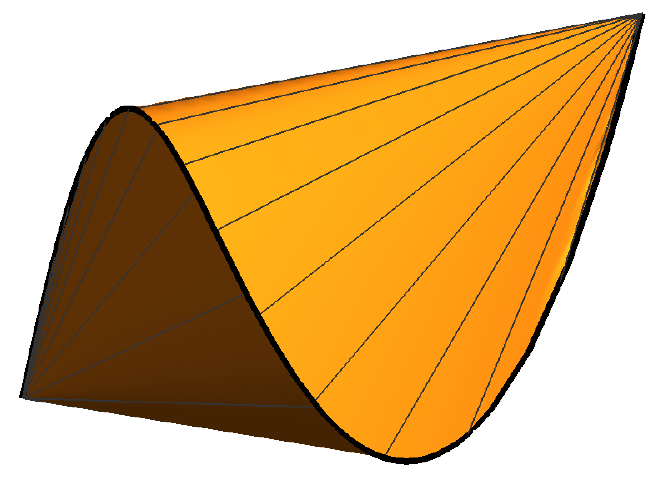}
\caption{The sets $M_1(A)$, $M_2(A)$, $M_3(A)$ in $\{m_0=1\}$ for $A = \{0,2,5,9\}$.}
\label{fig:C5ex}
\end{figure}

\section{Increasing and decreasing step functions}\label{sec:Monotone}

In this section, we study the moment cones of non-negative monotone
functions on the unit interval $[0,1]$. 
We define the increasing and decreasing moment cones
\begin{align*}
M^{\uparrow}(A) & = \overline{\left\{ \left(\int_0^1x^af(x)dx\right)_{a\in A} : f \text{ is nonnegative and {\em increasing} on }[0,1] \right\}} 
\text{ and } \\
M^{\downarrow}(A) &= \overline{\left\{ \left(\int_0^1x^af(x)dx\right)_{a\in A} : f \text{ is nonnegative and {\em decreasing} on }[0,1] \right\}}. 
\end{align*}
Recall that if a function $f:[0,1]\to \R$ is monotone, then it is automatically 
Borel-measurable. 
As in the non-monotone case, all of these moment vectors can be achieved as  
a limit of moments of step functions with a bounded number of steps. 
For $k\in \N$, let $S_k^{\uparrow}$ denote the set of nonnegative, \emph{increasing} step functions on $[0,1]$ with at most $k$ discontinuities. Similarly, let $S_k^{\downarrow}$ denote the analogous set of \emph{decreasing} step functions.  This corresponds to requiring $ y_1\leq y_2\leq \hdots \leq y_{k+1}$ or $y_1\geq y_2\leq \hdots \geq y_{k+1}$ in \eqref{eq:stepFunction}.

Similarly, for finite $A\subset \N$, we consider the $A$-moments of these step functions, 
\[
M_k^{\square}(A) \ =\ \overline{\left\{ \left(\int_0^1x^af(x)dx\right)_{a\in A} : f\in S_k^{\square}\right\}} 
\ \ \text{ for } \ \ 
\square\in \left\{ \uparrow, \downarrow\right\}.
\]
Just as with $M_k(A)$, we see that the set $M_k^{\square}(A)$ is invariant under nonnegative scaling,
$M_k^{\square}(A)\subseteq M_{\ell}^{\square}(A)$ when $k \leq\ell$ 
and $M_{k}^{\square}(A) + M_{\ell}^{\square}(A) \subseteq
M_{k+\ell}^{\square}(A)$.

As in the non-monotone case, we can understand the cones $M^{\square}(A)$ as the conical hull of curve segments. 
\begin{Def}\label{def:curve_steps}  
We define maps $\gamma^{\uparrow}_A$ and $\gamma^{\downarrow}_A$ from $[0,1]$ to $\R^A$ 
where, for $t\in [0,1]$, $\gamma^{\uparrow}_A(t)$ and $\gamma^{\downarrow}_A(t)$
are the $A$-moment vectors of the step functions $(1/(1-t)){\bf 1}_{(t,1]}$ and $(1/t^{ \min(A)+1}){\bf 1}_{[0,t]}$, respectively. For every $a\in A$, the $a$th coordinate of these maps are given by 
\[ \begin{array}{cccrlcrl}
& \left(\gamma^{\uparrow}_A(t)\right)_{a} & = & \dfrac{1}{1-t}&\displaystyle\int_t^1 x^a dx & = &
\dfrac{1}{a+1} & \displaystyle\sum_{i=0}^at^i \\[5mm]  \text{ and } &
\left(\gamma^{\downarrow}_A(t)\right)_a &  = & \dfrac{1}{t^{\min(A)+1}} &\displaystyle\int_0^t x^adx & = & \dfrac{1}{a+1} & t^{\: a-\min(A)}. \end{array}
\]
\end{Def}
We observe that  $\gamma^{\uparrow}_A(0)=\gamma^{\downarrow}_A(1)=
\left(1/(a+1)\right)_{a\in A}$ corresponds to the
  moment vector of constant function ${\bf 1}_{[0,1]}$. 
  The other end points correspond to point masses. Specifically, 
  $\gamma^{\uparrow}_A(1) = v_A(1)$ is the moment vector of a point mass at $t=1$ and $\gamma^{\downarrow}_A(0) = \frac{1}{ \min(A)+1}v_B(0)$ for $B = \{a- \min(A): a\in A\}$ corresponds to a point mass at $t=0$.

   \begin{remark}\label{rem:MonotoneClosure}
     The conical hull over $\{\gamma_A^{\square}(t): t\in [0,1]\}$  is
     closed because this curve is compact and does not contain the origin. 
     Indeed, for $\square = \uparrow$, the $a$th coordinate of
     $\gamma^{\uparrow}_A(t)$ is $\geq (1/a+1)$ for all $t$.
     For $\square = \downarrow$, the $\min(A)$-th coordinate of
     $\gamma^{\downarrow}_A(t)$ is identically $1/(\min(A)+1)$.
   \end{remark}

\begin{lem}\label{lem:MonotoneConicalHull}
  For $\square\in\{\uparrow, \downarrow\}$, the cone $M^{\square}(A)$ equals 
  the conical hull of $\{\gamma_A^{\square}(t):t\in [0,1]\}$.
\end{lem}
\begin{proof}
Since $M^{\square}(A)$ is a convex cone containing the point 
$\gamma_A^{\square}(t)$ for all $t$, it automatically contains the conical hull 
of this curve.  

For the other direction, consider a monotone function $f:[0,1]\to \R$. 
  We can  construct  a sequence of step functions 
  $f_n$ converging uniformly to $f$ on $[0,1]$. 
  For example, we may 
  take $f_n = \sum_{i=1}^{n} \frac{M}{n}{\bf 1}_{T_i}$
  where $M\in \{f(0), f(1)\}$ is the maximal value of $f$ on $[0,1]$ and 
  ${\bf 1}_{T_i}$ is the indicator function of  $T_i =  \{x\in [0,1]: f(x)\geq i M/n\}$.  
  That is $f_n(x) = \frac{M}{n} \cdot \lfloor \frac{n}{M} f(x) \rfloor$. 
 Note that $|f_n -f| \leq M/n$ and so $f_n$ converges uniformly to $f$ on $[0,1]$. 
 It follows that for any $a$, $x^a f_n$ converges uniformly to $x^a f$ and so 
 the integral $\int_0^1 x^a f_n(x)dx$ converges to $\int_0^1 x^af(x)dx$.
 
 Note that the set $T_i$ defined above has the form $(s_i,1]$ or $[s_i,1]$ if $f$ is increasing and 
 $[0,s_i]$ or $[0,s_i)$ if $f$ is decreasing for some $s_i\in [0,1]$. 
 The moment vector of $f_n$ therefore is a conic combination 
 of the points $\gamma_A^{\square}(s_i)$ for the appropriate $\square\in \{\uparrow, \downarrow\}$. 
 Taking $n\to \infty$ shows that the moment vector of $f$ 
 belongs to the closure of the conical hull of $\{\gamma_A^{\square}(t):t\in [0,1]\}$. 
 
 Therefore the moment cone 
 $\{ (\int_0^1 x^a f(x)dx)_{a\in A} : f\text{ nonnegative and increasing on }[0,1]\}$ belongs to the closure of the conical hull of $\{\gamma_A^{\uparrow}(t):t\in [0,1]\}$.  By definition, 
 $M^{\uparrow}(A)$ is the closure of this set and so also belongs to the closure of 
 this conical hull. Similarly $M^{\downarrow}(A)$ belongs to the closure of the conical hull of $\{\gamma_A^{\downarrow}(t):t\in [0,1]\}$.
 By Remark~\ref{rem:MonotoneClosure}, both of these conical hulls are 
 already closed.
 \end{proof}

\begin{prop}\label{prop:MonotoneCone}
  If $k\geq \left
    \lfloor{\frac{|A|}{2}}\right \rfloor$, then we have
  $M_k^{\uparrow}(A)=M^{\uparrow}(A)$ and $M_k^{\downarrow}(A)=M^{\downarrow}(A)$.
\end{prop}
\begin{proof} Our proof proceeds similarly to that of Theorem~\ref{thm:MomentBound1}.
Let ${\bf m}$ be a point of the boundary of $M^{\square}(A)$. We want to
  express ${\bf m}$ as the $A$-moment of an increasing step function of the
  fewest steps possible. 
  Let $\ell:\R^A\to \R$ define a supporting hyperplane of $M^{\square}(A)$ at ${\bf m}$, 
  so that $\ell\geq 0$ on $M^{\square}(A)$ and $\ell({\bf m}) = 0$. 
  By Lemma~\ref{lem:MonotoneConicalHull}, $M^{\square}(A)$ is the conical hull of a curve, 
  hence ${\bf m}$ will lie in the conical hull of points on this curve with $\ell=0$. 
 We use this to show that ${\bf m}$ belongs to $M_k^{\square}(A)$  for $k\geq \left \lfloor{\frac{|A|}{2}}\right \rfloor$. \\

  ($\downarrow$) Let  
  $p(x) =\ell\left(\gamma^{\downarrow}_A(x)\right) =\displaystyle\sum_{a\in A} \frac{p_a}{a+1}x^{a-\min(A)} $. 
  The polynomial $p$ is nonnegative on $[0,1]$.  By Descartes' rule of signs, $p$ has at most $|A|-1$ positive roots,
  counting multiplicity, and if $p_{\min(A)}=0$, then it has at most $|A|-2$. 
  Let $i$ denote the number of distinct roots of $p$ in $(0,1)$ and $b=1$ if $p(0)=0$ and $0$ otherwise. 
  Since each interior root of $p$ must have multiplicity $\geq 2$, this gives $2i+b\leq |A|-1$. 
  Note that $\gamma^{\downarrow}_A(t)\in M^{\downarrow}_1(A)$ for all $t\in [0,1)$ and belongs to
  $M^{\downarrow}_0(A)$ for $t=1$.  Therefore ${\bf m}$ belongs to $M^{\downarrow}_{k}(A)$
  for $k= i+b \leq \frac{1}{2}(|A|-1+b)$.  The bound follows from the integrality of $i+b$ and $b\in \{0,1\}$.

  ($\uparrow$) Let  
  $p(x) =\ell\left(\gamma^{\uparrow}_A(x)\right) = \displaystyle\sum_{a\in A}\frac{p_a}{a+1}\displaystyle\sum_{i=0}^ax^i$, which is a polynomial nonnegative on $[0,1]$.
  Again, by Descartes' rule of signs, $p$ has at most $|A|-1$ positive roots,
  counting multiplicity. 
  If $i$ is the number of distinct roots of $p$ in $(0,1)$ and $b=0$ if $p(1)=0$ and $0$ otherwise, this 
  gives that  $2i+b\leq |A|-1$.  As before, 
  $\gamma^{\uparrow}_A(t)\in M^{\uparrow}_1(A)$ for all $t\in (0,1]$ and belongs to
  $M^{\uparrow}_0(A)$ for $t=0$.  Therefore ${\bf m}$ belongs to $M^{\downarrow}_{k}(A)$
  for $k= i+b \leq \frac{1}{2}(|A|-1+b)\leq \frac{1}{2}|A|$.

  Now consider ${\bf m}$ in the interior of $M^{\square}(A)$ and let ${\bf c}$ be the
  moment vector of the constant function ${\bf 1}_{[0,1]}$. Let $\lambda^*$ be the maximum
  value of $\lambda\in \R$ for which ${\bf m}-\lambda {\bf c}$ belongs to
  $M^{\square}(A)$. Since ${\bf m}\in M^{\square}(A)$, we know that
  $\lambda^*\geq 0$, and for sufficiently large $\lambda$, ${\bf m}-\lambda
  {\bf c}\notin M^{\square}(A).$ Thus ${\bf m}-\lambda^* {\bf c}$ belongs to the
  boundary of $M^{\square}(A)$, which is equal to the boundary of
  $M_k^{\square}(A)$ by the argument above. Hence, ${\bf m}$ also belongs to $M_k^{\square}(A)$.
\end{proof}

\begin{prop}\label{prop:monotoneLowerBound}
  For all $k< \left\lfloor{\frac{|A|}{2}}\right \rfloor$, the cone
$M_k^{\square}(A)$ is a proper subset of $M^{\square}(A)$.
\end{prop}

\begin{proof}
  The cone $M_k^{\square}(A)\subset \R^{|A|}$ is a conic combination of $k$ points
  on the boundary curve $\gamma^{\square}_A$, each contributing two
  degrees of freedom, and the point corresponding to the image of the constant
  step function
  $\gamma^{\uparrow}_A(0)=\gamma^{\downarrow}_A(1)$, contributing a single
  degree of freedom. Therefore, the semialgebraic set $M_k^{\square}(A)$ has dimension at most
  $\min\left\{2k+1, |A|\right\}$. 
  The cone $M^{\square}(A)$ is full-dimensional in $\R^{|A|}$.
  Let $n = \lfloor |A|/2\rfloor$ so that $|A|$ is $2n$ or $2n+1$. 
   In either case, we observe that for $k\leq n-1$, the dimension
 of $M_k^{\square}(A)$ is less than or equal to $2n-1$, hence it cannot
 fill up all of $M^{\square}(A)$.
\end{proof}

\begin{example}
For $A=\{0,2,5,9\}$, $M_1(A)$ is a union of $M^{\uparrow}_1(A)$ and $M^{\downarrow}_1(A)$,
shown on the left in Figure~\ref{fig:C5ex}. Since $1<  2 = \lfloor |A|/2 \rfloor$, 
these sets are not full dimensional and so cannot fill up $M^{\uparrow}(A)$ or $M^{\downarrow}(A)$. 
For $k=2 =\lfloor |A|/2 \rfloor$,  $M_2^{\uparrow}(A) = M^{\uparrow}(A)$ and $M_2^{\downarrow}(A) = M^{\downarrow}(A)$. These form parts of the full dimensional set $M_2(A)$ shown in the middle of Figure~\ref{fig:C5ex}.
\end{example}

\section{Connection with coalescence manifold}\label{sec:coalManifold}

The motivation for studying moments of step functions comes from the
field of {\bf population genetics}. A central problem in this area is:
\begin{question}
  Given a sample of $n$ genomes from a present-day population, what
  inferences can be drawn regarding the history of that population?
\end{question}

Our approach to the problem is to fix a function $\pop(t)$ describing
{\em effective population size} at time $t$ before the present.
We then compute, as a function of $\pop$, a vector of invariants
${\bf c}$ associated to the genome sample. Understanding the relationship
between $\pop$ and ${\bf c}$ will allow
us to infer likely values of $\pop$ based on measured data.

Following \cite{BS}, we model the natural process
of the production of a sample of $n$ genomes as follows:

\begin{itemize}
\item The genealogical tree connecting $n$ individuals will be
  formed by taking coalescence of each pair of lineages as a Poisson point process
  with rate parameter $1/\pop(t)$, where
$\pop(t)$ is the effective population size at time $t$ before present. (Heuristically, looking
at the previous generation and picking parents at random, there is a $1/\pop(t)$ chance
that two lineages will pick the same parent.)
\item After the tree is specified, mutations are distributed on the tree as a Poisson
  point process with constant rate relative to branch length. The infinite-sites model
  is used, so that repeated mutation at a given site is disallowed, 
  which is a good model for large genomes.
  \end{itemize}

\begin{Def}
  Fixing a population history, and defining the random process as above, we define
  random variables:
  \begin{itemize}
  \item The {\em sample frequency spectrum} (also known as the {\em site} or {\em allele frequency spectrum}), abbreviated SFS, is the vector of random variables $(X_{n,b})_{b = 1,\ldots,n-1}$ where $X_{n,b}$ denotes
    the number of mutations that are shared by exactly $b$ out of the $n$ individuals.
  \item The {\em coalescence vector} is the vector $(T_{i,i})_{i=2,\ldots,n-1}$ of the time
    at which a sample of size $i$ has exactly $i$ distinct lineages, i.e. the time until the
    first coalescence.
  \end{itemize}
For a fixed population function $\pop$,  taking expectations gives the population invariants
  $\xi_{n,b} = \mathbb{E}[X_{n,b}]$ and $c_i = \mathbb{E}[T_{i,i}]$.
\end{Def}
\noindent In practice, the SFS is more frequently discussed as a summary statistic, but the
coalescence vector is simpler to use in computations. Fortunately, Polanski and Kimmel \cite{KP}
proved that they are related by a linear transformation $A_n$, a matrix entirely determined
by sample size $n$. Therefore, we focus on the coalescence vectors $(c_i)$.

\begin{fact}
  We make the reasonable assumption that $\pop(t)$ is bounded below by $0$ and
  bounded above by a fixed $P$.
  Applying integration by parts and change of variables to the expected value
  of an exponential distribution yields the
  following expression for $c_i$ in terms of $\pop(t)$:
  \begin{equation}
    c_i(\pop) = \int_0^\infty \tilde{\pop}(\tau) \exp\left[-\binom{i}{2} \tau\right] \mathrm{d} \tau, \label{eq:c}
  \end{equation}
 where $\tilde{\pop}(\tau) = \pop(R_{\pop}^{-1}(\tau))$ and $R_{\pop}(t) = \int_0^t \frac{1}{\pop(x)}
 dx$. Because $0 < \pop (t) < P$, the function $R_{\pop}$ is strictly increasing and unbounded;
 thus, it is a bijection from $\mathbb{R}_{\geq 0} \to \mathbb{R}_{\geq 0}$, so the inverse is well-defined. We call $\tilde{\pop}(\tau)$ the {\em transformed population history}.
\end{fact}

The coalescence vector can thus be considered a function from the space of (bounded)
population history functions to $\mathbb{R}^{n-1}$. Since the former space is
infinite-dimensional and the latter is finite-dimensional, it is
natural to restrict
our attention to a finite-dimensional space of population history functions. A common
choice for this, motivated by injectivity considerations in \cite{BS}, 
is \[\tilde{S}_k = \{ \text{nonnegative step functions on  }\mathbb{R}_{\geq 0} \text{  with at most $k$
  breakpoints}\}.\]

\begin{Def}
  Let $n,k$ be integers with $ n \geq 2$ and $k \geq 0$. 
  The {\em coalescence manifold} $\mathcal{C}_{n,k}$ is the Euclidean closure of
  the set of vectors ${\bf \tilde{c}}(\pop) = {\bf c}(\pop)/||{\bf c}(\pop)||_1$ for all
  $\pop \in \tilde{S}_k$. Here, ${\bf c}(\pop) = (c_2(\pop),\ldots,c_n(\pop))$ where $c_i(\pop)$ is defined as in Equation~\ref{eq:c}.
  \label{def:coal}
\end{Def}
\noindent Because the vectors are normalized to
have sum one, the coalescence manifold lives in the simplex $\Delta^{n-1}$.
Note that this definition deviates slightly
from the definition in \cite{BRS} by allowing $k$ breakpoints instead of $k$ epochs
(i.e. constant intervals). This shifts the index down by one. We now connect
back to the moment cones studied above.

\begin{thm} \label{thm:coal}
Let $A = \{\binom{i}{2} - 1 : i=2, \hdots, n\}$.  
The coalescence manifold $\mathcal{C}_{n,k}$ equals the intersection of the 
cone $M_k(A)$ with the affine hyperplane of points with coordinate sum 
equal to one: 
\[\mathcal{C}_{n,k}  \ =\  \left\{{\bf m}\in M_{k}(A) \  :  \ \sum_{a\in A} m_a = 1\right\}.\]
\end{thm}

Before we prove the theorem, we demonstrate two lemmas that will simplify the proof.
\begin{lem}\label{lem:popTilde}
  Define $\tilde{\pop}(\tau)$ as in Equation~\ref{eq:c}.
  Then $\pop(t)\in \tilde{S}_k$ if and only if
  if $\tilde{\pop}(\tau) \in \tilde{S}_k$.
\end{lem}

\begin{proof}
  Let $0 = s_0 < \cdots < s_{k-1} < s_k$ be the sequence of breakpoints of $\pop(t)$.
  The function $R_{\pop}(t)$
  is a monotone increasing function, so the conditions below are equivalent:
  \[ s_j < t \leq s_{j+1} \iff R_{\pop}(s_j) < R_{\pop}(t) \leq R_{\pop}(s_{j+1}). \]
  Since $\pop$ is constant on $(s_j,s_{j+1}]$, the transformed history
    $\tilde{\pop}(\tau) = \pop(R_{\pop}^{-1}(\tau))$ is constant on $(R_{\pop}(s_j),R_{\pop}(s_{j+1})]$.
      This implies that there are still at most $k$ breakpoints.

      For the reverse direction, repeat the argument with $R_{\pop}^{-1}$ in place of $R_{\pop}$.
\end{proof}

\begin{lem}\label{lem:InverseTransform}
  Let $q$ be a strictly positive step function in $\tilde{S}_k$.
  Then, there exists $\pop$ in $\tilde{S}_k$ such that
  $q(\tau) = \pop(R_{\pop}^{-1}(\tau))$ where $R_{\pop}(t) =
  \int_0^t \frac{1}{\pop(x)} \mathrm{d}x$ as above.
\end{lem}

\begin{proof}
    Let $Q(t) = \int_0^t q(x) \mathrm{d}x$. We claim the desired function is
    $\pop(t) = q(Q^{-1}(t))$. First, note that because $q$ is strictly positive and takes only 
    finitely many values, it is bounded away from zero. Therefore $Q$ is strictly increasing 
    and takes all values in $[0,\infty)$. Its inverse $Q^{-1}$ therefore exists and is also increasing
    with range $[0,\infty)$. 
    It follows that $\pop$ takes the same values in the same order as $q$. 
    In particular, $\pop\in \tilde{S}_k$. 
    
    To check that $q(t) = \pop(R_\pop^{-1}(t))$,  we first show that 
    $R_\pop(Q(t)) = t$ for all $t\geq 0$.  By definition, 
    \[ R_\pop(Q(t)) \ = \   \int_0^{Q(t)} \frac{1}{\pop(x)} dx \ = \  \int_0^{Q(t)} \frac{1}{q(Q^{-1}(x))} dx \ = \ \int_0^{t} \frac{1}{q(w)} q(w) dw \ =\ t,\]
    where the penultimate equation comes from  substituting $x = Q(w)$ and $dx = q(w) dw$. 
  Since both $Q$ and $R_\pop$ are invertible, we see that $t = Q^{-1}(R_\pop^{-1}(t))$ for all $t$. 
  Applying $q$ to both sides then gives the claim.   \end{proof}

\begin{proof}[Proof of Theorem~\ref{thm:coal}]
  We show that the set of coalescence vectors coming from population
  histories in $\tilde{S}_k$ is equal to the
  set of moments in $M_k(A)$ summing to $1$. The equality of the two closures is then automatic.

  Assume $\pop \in \tilde{S}_k$. From Lemma~\ref{lem:popTilde}, $\tilde{\pop}$ is also in $\tilde{S}_k$.
  Starting with Equation~\ref{eq:c},
  we substitute $u = e^{-\tau}$ to obtain:
  \[ c_i(\pop) =  \int_{0}^{1} \tilde{\pop}^*(u) u^{\binom{i}{2}-1}  \mathrm{d}u, \text{  where  } \tilde{\pop}^*(u) = \pop(R_{\pop}^{-1}(-\ln(u))).\]
  The function $\tilde{\pop}^*$ is piecewise-constant on $[0,1]$ with at most $k$
  breakpoints, so is in $S_k$; therefore, the quantity $c_i$ is the $(\binom{i}{2}-1)$-th
  moment of $\tilde{\pop}^*$. This implies that {\bf c} is in $M_k(A)$ where $ A=\{\binom{i}{2} -1 : i = 2,\ldots,n\}$. 
  Normalizing ${\bf c}$ is equivalent to scaling $\tilde{\pop}^*$ so
  we may assume its sum is already equal to $1$.

  Conversely, up to closure, 
  any moment vector in $M_k(A)$ summing to $1$ comes from some $f \in S_k$.
  Changing our domain to $\mathbb{R}_{\geq 0}$ gives $q(\tau) = f(e^{-\tau})$ in $ \tilde{S}_k$.
  By Lemma~\ref{lem:InverseTransform}, we can produce $\pop \in \tilde{S}_k$ that gives
  transformed population history $q$.
 \end{proof}
 
 \begin{figure}[h!]
\includegraphics[height=1.25in]{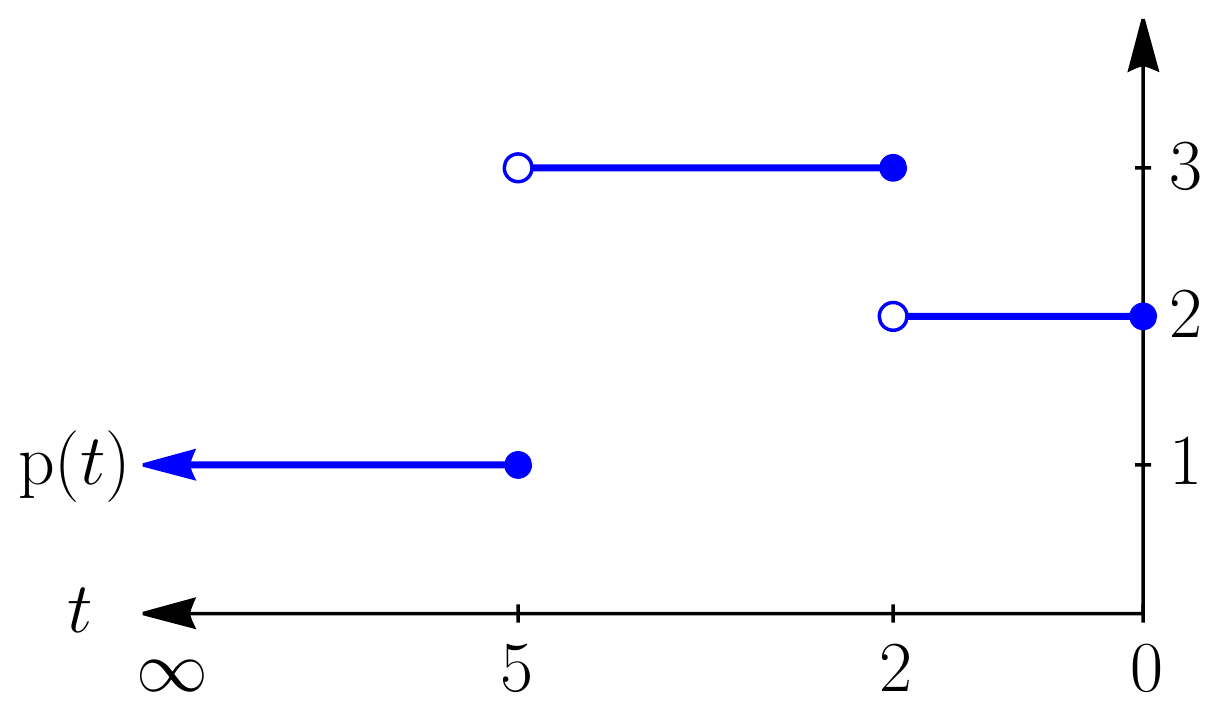} \quad \includegraphics[height=1.25in]{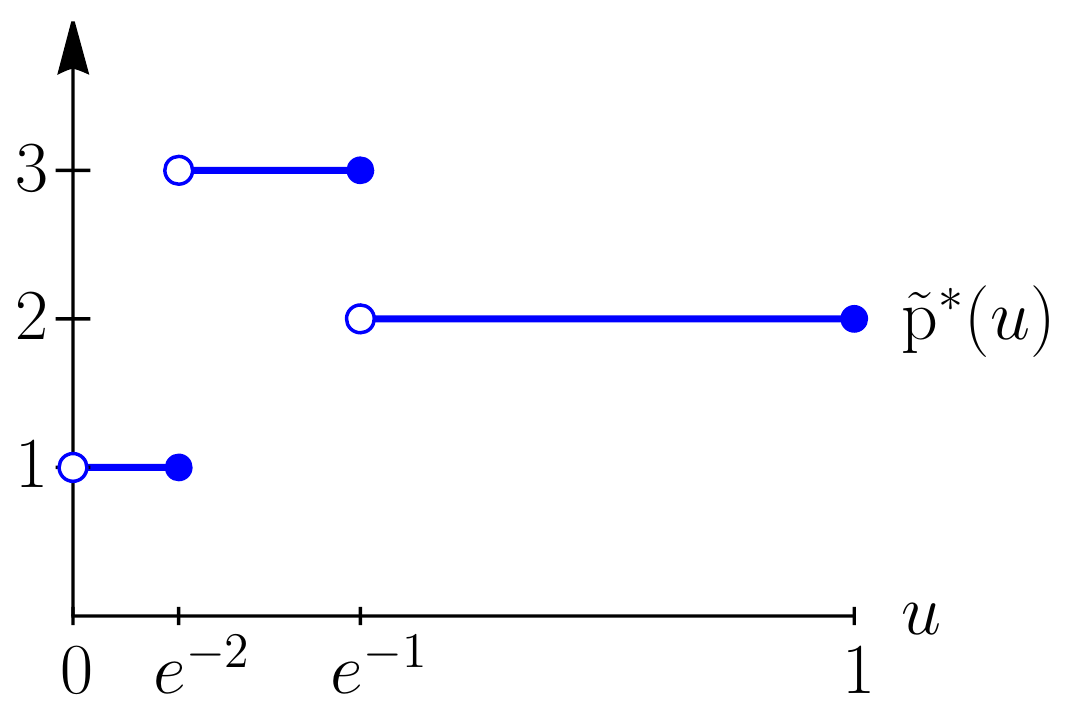}  \includegraphics[height=1.25in]{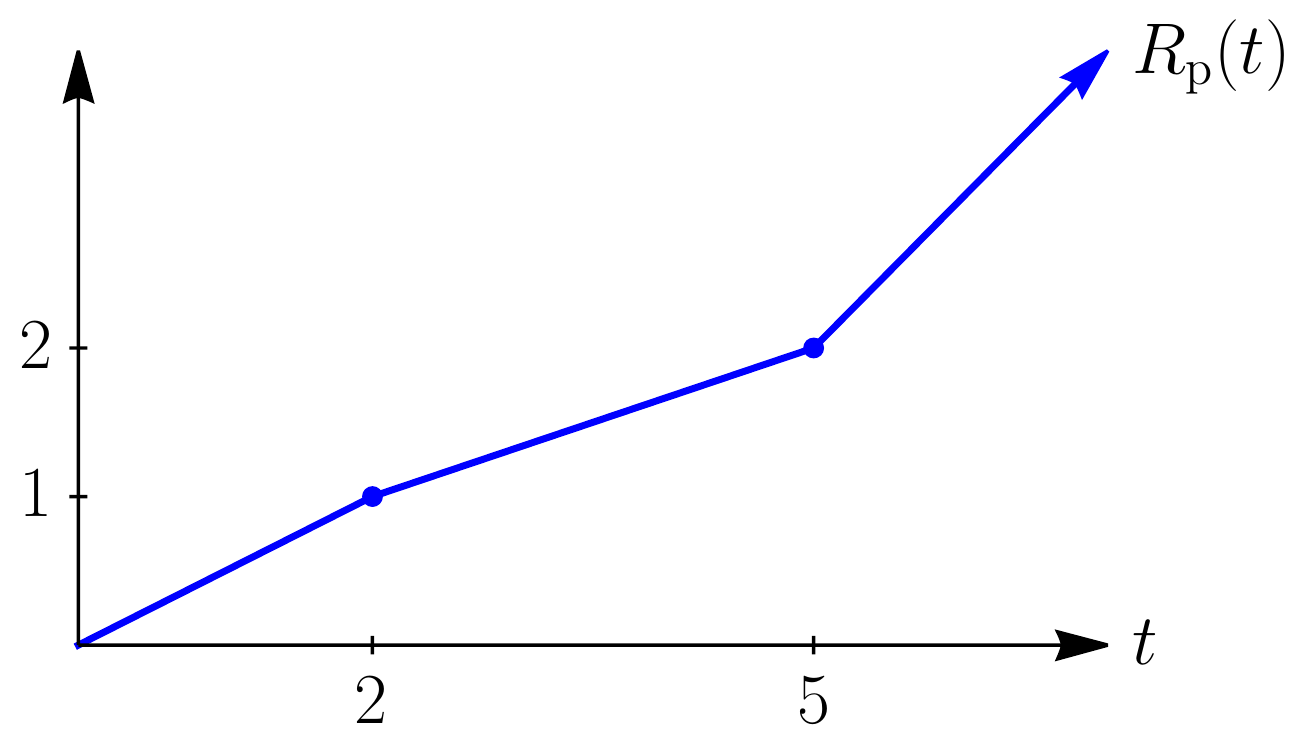}
\caption{The functions $\pop$, $\tilde{\pop}^*$, and $R_{\pop}$ 
from Example~\ref{ex:etas}. } \label{fig:etas}
\end{figure}

 \begin{example}\label{ex:etas}
 Consider the  population function 
 $\pop(t) = p_1\cdot{\bf 1}_{[0,b_1)}+p_2\cdot {\bf 1}_{[b_1,b_2)}+p_3\cdot{\bf 1}_{[b_2,\infty)}$ where $p_1, p_2, p_3,b_1, b_2\in \R_{> 0}$ with $b_1<b_2$. 
The function $R_{\pop}(t)$ is piecewise linear, given by
 \[R_{\pop}(t) = \int_0^t \frac{1}{\pop(x)}dx =
  \frac{t}{p_1}{\bf 1}_{[0,b_1)}+ \left(\frac{t-b_1}{p_2}+\frac{b_1}{p_1}\right){\bf 1}_{[b_1,b_2)}+\left(\frac{t-b_2}{p_3} + \frac{b_2-b_1}{p_2}+\frac{b_1}{p_1}\right){\bf 1}_{[b_2,\infty)}.\] This function is unbounded and strictly increasing with $R_{\pop}(0) = 0$,  so it has an inverse $R_{\pop}^{-1}$ that is also increasing and unbounded on $\R_{\geq 0}$. The function  $\tilde{\pop}(\tau) = \pop(R_{\pop}^{-1}(\tau)) $ is still piecewise constant with two break points $R_{\pop}(b_1)=b_1/p_1$ and $R_{\pop}(b_2)=(b_2-b_1)/p_2+b_1/p_1$, obtained by solving $ R_{\pop}^{-1}(\tau)=b_i$. The $i$th entry of the coalescence vector is then 
  \[c_i = \int_0^{\infty}\tilde{\pop}(\tau) e^{-\binom{i}{2}\tau}d\tau
  =  \int_0^{1} \tilde{\pop}^*(u)u^{\binom{i}{2}-1}du \ \ \text{ where } \ \ \tilde{\pop}^*(u) = \tilde{\pop}(-\ln(u)).\]
The second equality comes from the change of coordinates $u = e^{-\tau}$.
 Note that $\tilde{\pop}^*$ is the step function given by  
 \[
 \tilde{\pop}^*= p_3\cdot{\bf 1}_{(0,s_1]}+ p_2\cdot {\bf 1}_{(s_1,s_2]}+p_1\cdot {\bf 1}_{(s_2,1]}
\ \ \text{ where } \ \ 
s_1 = e^{-R_{\pop}(b_2)}\text{ and }s_2 =  e^{-R_{\pop}(b_1)}.
\]
The graphs of $\pop$ and $\tilde{\pop}^*$ for the values $(p_1, p_2, p_3) = (2,3,1)$ and $(b_1, b_2) = (2,5)$
are shown in Figure~\ref{fig:etas}.  In this case, the break points of  $\tilde{\pop}^*$ are $e^{-R_{\pop}(b_2)}=e^{-2}$ and $e^{-R_{\pop}(b_1)}=e^{-1}$.
 \end{example}

 \begin{remark}
  Note that because $\pop(t)$ denote the population size at time $t$ \emph{before} the present, 
 a population increasing over time corresponds to the function $\pop(t)$ decreasing as a function of $t$, i.e. 
$p_1>p_2>p_3$ in the example above. Note that $\pop(t)$ is decreasing in $t$ if and only if $\tilde{\pop}(\tau)$ is decreasing in $\tau$. The parametrization $u = e^{-\tau}$ reverses direction and so the 
function $\tilde{\pop}^*(u)$ is then increasing as a function of $u$. In these coordinates, $u=0$ corresponds ``infinitely long ago'' ($t=\infty$) and $u=1$ corresponds to the present ($t=0$). 
Therefore coalescence vectors of populations growing over time are moments of increasing step functions on $[0,1]$. 
 \end{remark}

Theorem~\ref{thm:coal} allows us to apply our results from $M_k(A)$ to $\mathcal{C}_{n,k}$.

\begin{cor}\label{cor:coalBound} 
$ \mathcal{C}_{n,n-2} = \mathcal{C}_{n,k}$ for all $k \geq n-2$ and 
$ \mathcal{C}_{n,n-3} \subsetneq \mathcal{C}_{n,n-2}$. 
\end{cor}
\begin{proof}
For $A = \left\{\binom{i}{2} : i=2, \hdots, n\right\}$, $|A|$ equals $n-1$. 
By Theorem~\ref{thm:MomentBound1}, $M_{k}(A)= M(A)$ for all $k\geq n-2$. 
In particular, $M_{n-2}(A)= M_k(A)$ for all $k\geq n-2$. 
Intersecting with the hyperplane $\{{\bf m}   :  \sum_{a\in A} m_a = 1\}$ gives that 
$\mathcal{C}_{n, n-2}= \mathcal{C}_{n,k}$ for all $k\geq n-2$. 
By Corollary~\ref{cor:MomentBound2}, $M_k(A)\neq M(A)$ for $k<|A|-1 = n-2$. Hence $M_{n-3}(A) \neq M(A)$. Since $M(A) = M_{n-2}(A)$, 
 intersecting with the hyperplane $\{{\bf m}   :  \sum_{a\in A} m_a = 1\}$ gives that 
$\mathcal{C}_{n, n-2}\neq \mathcal{C}_{n,n-3}$. 
\end{proof}

Affine transformations the sets $\mathcal{C}_{5, 1}$, $\mathcal{C}_{5, 2}$ 
and $\mathcal{C}_{5, 3}$ are show in Figure~\ref{fig:C5ex}.
As promised, $\mathcal{C}_{5, 3}$ is convex and $\mathcal{C}_{5, k}$ is a strict subset for $k<3$.

\section{Connections with semidefinite programming}\label{sec:sdp}

In this section, we describe how to write the moment cone $M(A)$ 
and coalescence manifold $\mathcal{C}_{n,n-2}$ as projections of spectrahedra. 
This gives rise to natural algorithms for testing membership 
and finding nearest points in these sets based on semidefinite programming. 
Formally, a spectrahedron is a set of the form 
$\{{\bf x}\in \R^n : A_0 + \sum_{i=1}^n x_i A_i \succeq 0\}$ where
 $A_0, \hdots, A_n$ are real symmetric matrices and 
 $X\succeq 0$ denotes that the 
matrix $X$ is positive semidefinite. These are the feasible sets of
semidefinite programs. See e.g.~\cite[Ch. 5 and 6]{ConvexAlgBook}.
Python code for computing the nearest point in $\mathcal{C}_{n,n-2}$ to an 
arbitrary point in $\R^{n-1}$ is available at:
\centerline{\texttt{https://github.com/gescholt/DistanceToCoalescenceManifold}}

\begin{thm}[Theorems 10.1 and 10.2 \cite{Sch}]\label{thm:Spectrahedron}
For any $d\in \Z_+$, the cone $M(\{0,1,\hdots, d\})$ is a spectrahedron. If $d = 2e$ is even, then 
\[
M(\{0,1,\hdots, d\}) = \left \{ {\bf m}\in \R^{d+1} : (m_{i+j})_{0\leq i, j \leq e}\succeq 0 \text{ and } 
\left(m_{i+j+1}-m_{i+j+2}\right)_{0\leq i, j \leq e-1}\succeq 0\right\}, 
\]
and if $d = 2e+1$ is odd, then 
\[
M(\{0,1,\hdots, d\}) = \left \{ {\bf m}\in \R^{d+1} :  (m_{i+j+1})_{0\leq i, j \leq e}\succeq 0 \text{ and }  (m_{i+j}-m_{i+j+1})_{0\leq i, j \leq e}\succeq 0\right \}.
\]
\end{thm}

\begin{cor}\label{cor:M(A)ProjSpec}
For any finite set of integers $A\subset \N$, 
the convex cones $M(A)$, $M^{\uparrow}(A)$ and $M^{\downarrow}(A)$ 
are projections of the spectrahedron $M(\{0,1,\hdots, \max(A)\})$.
\end{cor}
\begin{proof}
Let $d = \max(A)$. 
Note that by definition, $M(A)$ equals the closure of the projection of $M(\{0,1,\hdots, d\})$ under the map 
$(m_0, m_1, \hdots, m_d) \mapsto (m_a)_{a\in A}$.  For $0\in A$, this projection is closed and 
otherwise, we replace $A$ with $B = \{a-\min(A): a\in A\}$ as in Lemma~\ref{lem:moment_shift}. 
By Theorem~\ref{thm:Spectrahedron}, $M(\{0,1,\hdots, d\})$ 
is a spectrahedron. 

More generally, consider any finite collection of polynomials $p_1, \hdots, p_n \in \R[x]_{\leq d}$.
We claim that the conical hull of the curve 
parameterized by ${\bf p}(t) = (p_1(t), \hdots, p_n(t))$ for $t\in [0,1]$ 
is the image of $M(\{0,1,\hdots, d\})$ under a linear map. 
Specifically, 
consider the linear map $\pi: \R^{d+1}\to \R^n$ taking $(m_0, m_1, \hdots, m_d)$
to $(\sum_{j=0}^d p_{ij}m_j)_{i\in [n]}$ where $p_i(x) = \sum_{j=0}^d p_{ij}x^j$. 
For any $t\in [0,1]$,  ${\bf p}(t)$ equals $\pi(v_d(t))$ where $v_d(t) = (1,t,t^2,\hdots, t^d)$. 
Since $M(\{0,1,\hdots, d\})$ is the conical hull of $\{v_d(t): t\in [0,1]\}$, 
the conical hull of $\{{\bf p}(t): t\in [0,1]\}$ is the image of $M(\{0,1,\hdots, d\})$ under $\pi$. 

Note that the coordinates of both $\gamma_A^{\uparrow}(t)$ and $\gamma_A^{\downarrow}(t)$ 
are given by polynomials in $t$ of degree $\leq d$. 
Then by Lemma~\ref{lem:MonotoneConicalHull} and the arguments above, 
both $M^{\uparrow}(A)$ and $M^{\downarrow}(A)$ can be written as the image of 
$M(\{0,1,\hdots, d\})$ under a linear map. 
\end{proof}

\begin{example}\label{ex:M(A)Spec}
For $A = \{0,2,5,9\}$, we write $M(A)$, $M^{\uparrow}(A)$ and $M^{\downarrow}(A)$ 
are projections of the spectrahedron $M(\{0,1,\hdots,9\})$.
By Theorem~\ref{thm:Spectrahedron}, this is given by the set 
of ${\bf m} = (m_0, \hdots, m_9)$ in $ \R^{10}$ for which the matrices 
\[
\begin{pmatrix}
m_1 & m_2 & m_3 & m_4 & m_5 \\
m_2 & m_3 & m_4 & m_5 & m_6\\
m_3 & m_4 & m_5 & m_6 & m_7\\
m_4 & m_5 & m_6 & m_7 & m_8\\
m_5 & m_6 & m_7 & m_8 & m_9\\
\end{pmatrix} \text{ and } 
\begin{pmatrix}
m_0-m_1 & m_1-m_2 & m_2-m_3 & m_3-m_4 & m_4-m_5  \\
m_1-m_2 & m_2-m_3 & m_3-m_4 & m_4-m_5 &m_5-m_6\\
m_2 - m_3 & m_3-m_4 & m_4-m_5 & m_5-m_6 &m_6-m_7  \\
m_3-m_4 & m_4-m_5 & m_5-m_6 & m_6-m_7 & m_7-m_8\\
m_4-m_5 & m_5-m_6 & m_6-m_7 & m_7-m_8 & m_8-m_9\\
\end{pmatrix}\]
are positive semidefinite. 
We obtain $M(A)$ as the image of this cone under the linear map 
${\bf m}\mapsto (m_0, m_2, m_5, m_9)$. 
Similarly, the cones $M^{\uparrow}(A)$ and $M^{\downarrow}(A)$ are the images of 
$M(\{0,1,\hdots,9\})$ under the (respective) maps
\[
{\bf m}\mapsto \left(m_0,\ \frac{m_0+m_1+m_2}{3}, \ \dfrac{1}{6}\displaystyle\sum_{i=0}^5m_i,\ \dfrac{1}{10}\sum_{i=0}^9 m_i \right)
\ \ \text{ and } \ \ 
{\bf m}\mapsto \left(m_0, \frac{m_2}{3}, \frac{m_5}{6}, \frac{m_9}{10}\right). 
\]
\end{example}

\begin{cor}\label{cor:M(A)test}
Testing membership any of the cones $M(A)$, $M^{\uparrow}(A)$ or $M^{\downarrow}(A)$
is equivalent to testing the feasibility of 
a semidefinite program in $\leq d+1$ variables with two matrix constraints, each of size $\leq d/2+1$, 
where $d = \max(A)$. 
\end{cor}

\begin{cor}
For $k\geq n-2$, the coalescence manifold  $\mathcal{C}_{n,k}$ 
is the projection of a spectrahedron. Testing membership in  $\mathcal{C}_{n,k}$ 
is equivalent to testing the feasibility of a semidefinite program in $\leq n^2/2$ variables with two matrix constraints, each of size $\leq n^2/4$. 
\end{cor}

\begin{proof} 
By Theorem~\ref{thm:coal} and Corollary~\ref{cor:coalBound}, for all $k\geq n-2$, 
coalescence manifold  $\mathcal{C}_{n,k}$ equals in the intersection of 
$M(A)$ with the affine hyperplane given by $\sum_{a\in A}m_a =1$ 
where $A = \{\binom{i}{2}-1: i=2,\hdots, n\}$. 
By Corollary~\ref{cor:M(A)ProjSpec}, $M(A)$ is the projection of 
$M(\{0,1,\hdots, d\})$ where $d = \binom{n}{2}-1$. 
It follows that $\mathcal{C}_{n,k}$ is the projection of the points in $M(\{0,1,\hdots, d\})$ satisfying the 
affine linear equation  $\sum_{a\in A}m_a =1$. 
The intersection of a spectrahedron with an affine linear space is again a spectrahedron and 
so $\mathcal{C}_{n,k}$ is the projection of a spectrahedron. 

The spectrahedron $M(\{0,1,\hdots, d\})$ is defined by two linear matrix inequalities of 
size $\leq d/2+1 \leq n^2/4$. There are at most $d+1=\binom{n}{2}\leq n^2/2$ variables.
\end{proof}

Similarly, given a point ${\bf p}\in \R^{n-1}$, we can use a semidefinite program to find the nearest point in 
$\mathcal{C}_{n,k}$ for sufficiently large $k$.  
This comes from the description of $\mathcal{C}_{n,k}$ above and the fact that distance minimization 
can be phrased as a semidefinite program (see, e.g. \cite{CVXbook}). 
Specifically,  
given $\x\in \R^{n-1}$, the matrix $n\times n$ matrix 
$\begin{pmatrix} \lambda & \x -\p \\ (\x-\p)^T  & {\rm Id}_{n-1} \end{pmatrix}$ is positive semidefinite 
if and only if $||\x -\p||^2_2 \leq \lambda$, where ${\rm Id}_{n-1} $ denotes the $(n-1)\times (n-1)$ identity matrix. 
Given a set $S \subset \R^{n-1}$, suppose that $\lambda^*$ and $\x^*$ obtain the minimum 
\[
\min_{\lambda\in \R, \x\in S}  \lambda  \ \ \  \text{ such that } \ \ \ \begin{pmatrix} \lambda & \x -\p \\ (\x-\p)^T  & {\rm Id}_{n-1} \end{pmatrix} \succeq 0. 
\]
Then $\x^*$ is (one of) the nearest points in $S$ to $\p$ and the distance $||\x^* - \p||_2$ is $ \sqrt{\lambda^*}$. 
In particular, if the set $S$ is the projection of a spectrahedron, then this minimization problem is a semidefinite program. 

\begin{cor} Given $\p\in \R^{n-1}$, the problem of finding the closest point to $\p$ in 
$\mathcal{C}_{n,k}$ for sufficiently large $k$ is equivalent to solving 
a semidefinite program in $\leq n^2/2$ variables with three matrices of size $\leq n^2/4$. 
\end{cor}

\begin{example} 
For $n=5$ and $k\geq 3$, $\mathcal{C}_{5,k}$ equals the set of points in $M(\{0,2,5,9\})$ 
with $m_0+m_2+m_5+m_9 = 1$.  Projecting from $M(\{0,1,\hdots, 9\})$, we see that 
\begin{align*}
\mathcal{C}_{5,k}= 
\biggl\{
&(m_0,m_2,m_5,m_9)\in \R^4 \ : \  m_0+m_2+m_5+m_9 = 1\\ & \text{ and }  \exists (m_1,m_3,m_4,m_6,m_7,m_8)\in\R^6  \text{ such that } 
(m_j)_{j=0,\hdots, 9} \in M(\{0,1,\hdots, 9\})
\biggl\}
\end{align*}
Let $\mathcal{A}({\bf m})$ and $\mathcal{B}({\bf m})$ denote the 
two $5\times 5$ matrices appearing in Example~\ref{ex:M(A)Spec}. 
Then $M(\{0,1,\hdots, 9\})$ is the set of points ${\bf m}\in \R^{10}$ for which 
$\mathcal{A}({\bf m})\succeq 0$ and $\mathcal{B}({\bf m})\succeq 0$.  
Given a point $\p = (a,b,c,d)\in \R^4$, we can find the closest point in $\mathcal{C}_{5,k}$
by solving the following semidefinite program with $10$ parameters and three $5\times 5$ linear matrix constraints: 
\begin{align*}
\min_{\lambda, m_0, \hdots, m_9} \lambda \   \text{ such } &\text{that }  \ m_0+m_2+m_5+m_9 = 1, \  \mathcal{A}({\bf m})\succeq 0, \ \mathcal{B}({\bf m})\succeq 0, \\
& \text{ and } \begin{pmatrix}
\lambda & m_0-a & m_2-b & m_5-c & m_9-d \\
m_0-a  & 1 & 0 & 0 & 0\\
m_2-b & 0 & 1 & 0 & 0 \\
m_5-c & 0 & 0 & 1 & 0 \\
m_9-d & 0 & 0 & 0 & 1
\end{pmatrix} 
\succeq 0.
\end{align*}
If $(\lambda^*, {\bf m}^*)$ denotes the points achieving this minimum, then 
$(m_0^*, m_2^*, m_5^*, m_9^*)$ is the closest point in $\mathcal{C}_{5,k}$ to ${\bf p}$ 
with distance $\sqrt{\lambda^*}$. 
\end{example}

\section{Discussion and open questions}\label{sec:questions}

One takeaway from Section~\ref{sec:momentIntro} is that the points on the boundary 
of $\mathcal{C}_{n,k}$ for $k\geq n-2$ correspond to moment vectors of point 
evaluations on $[0,1]$.  However these do not correspond to biologically meaningful population functions! 
Similarly, a point in the interior of $\mathcal{C}_{n,k}$ can come from several different population functions, some of which are more biologically plausible than others. 
One natural question from this standpoint is how to pick the right population history from the fiber of a coalescence vector.  
\begin{question}
Given a point ${\bf m}$ in the interior of $M_k(A)$, how can we find the ``best'' step function 
$f\in S_k$ with moment vector ${\bf m}$? 
\end{question}

Here there is some natural flexibility in the notion of ``best''. Ideally it should be biologically plausible and also easy to compute.  For plausibility, it might be reasonable to try to bound 
or minimize the ratios $y_{i+1}/y_{i}$ of consecutive population sizes. 
One step towards this would be to understand the structure of the fibers of the 
moment map $\mu_A$.  

\begin{question}
Are all fibers of the map $\mu_A:P \rightarrow \R^A$ given in equation \eqref{eq:muA} connected? 
\end{question}

For $k=2$ and $A = \{0,2,5\}$, the $(s_1,s_2)$-coordinates of the fibers of some points in $M_2(A)$
are shown below.

\begin{figure}[h!]
  \includegraphics[height=4in]{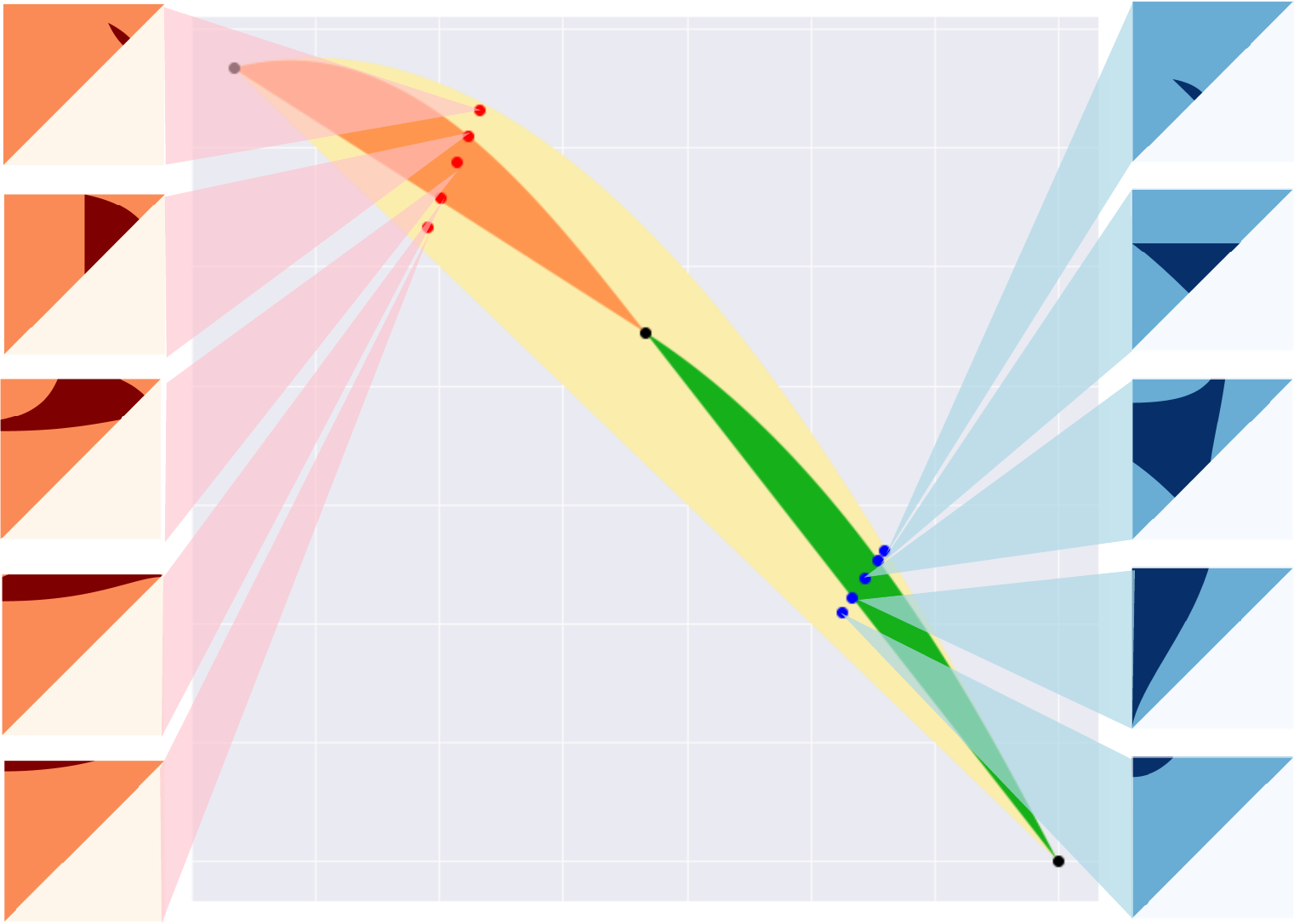}
  \caption{The central image depicts $M_2(A)$ in yellow. The orange region is $M^{\uparrow}(A)$ and the
    green region $M^{\downarrow}(A)$; their union is $M_1(A)$. The triangle above each point depicts the
    fiber as a subset of the $(s_1,s_2)$-simplex.}
  \end{figure}

To understand the fibers, it may also help to relate the combinatorial structure of the 
polytope $P$ (which is a product of two $k$-dimensional simplices) to the 
semi-algebraic and combinatorial structure of $M_k(A)$. For example, the boundary 
of $M_2(\{0,2,5,9\})$, seen in Figure~\ref{fig:C5ex}, comes from some of the two-dimensional faces of the four-dimensional polytope $P$. 

\begin{question} How does the facial structure of $P$ relate to the algebraic boundary of $M_k(A)$? \end{question}

Finally, Section~\ref{sec:sdp} gives an algorithm for testing membership in $M(A)$, 
which coincides with $M_k(A)$ for $k\geq |A|-1$.  It would be desirable to be able 
to test membership for smaller $k$ as well. 

\begin{question}
Is there an effective method to test membership in $M_k(A)$ for $k< |A|-1$? 
\end{question}

These sets are not convex and may have 
complicated semialgebraic structure (Figure~\ref{fig:C5ex}). 

\bibliographystyle{plain}


\end{document}